\documentclass[oneside,english]{amsart}
\usepackage[T1]{fontenc}
\usepackage[latin9]{inputenc}
\usepackage{babel}
\usepackage{amstext}
\usepackage{amsthm}
\usepackage{amssymb}
\usepackage[unicode=true,pdfusetitle,
 bookmarks=true,bookmarksnumbered=false,bookmarksopen=false,
 breaklinks=false,pdfborder={0 0 1},backref=false,colorlinks=false]
 {hyperref}

\makeatletter
\numberwithin{equation}{section}
\theoremstyle{plain}
\newtheorem{thm}{\protect\theoremname}
\theoremstyle{theorem}
\newtheorem{mainthm}[thm]{\protect\mainthmname}
\theoremstyle{remark}
\newtheorem{rem}[thm]{\protect\remarkname}
\theoremstyle{remark}
\newtheorem*{acknowledgement*}{\protect\acknowledgementname}
\theoremstyle{definition}
\newtheorem{defn}[thm]{\protect\definitionname}
\theoremstyle{plain}
\newtheorem{fact}[thm]{\protect\factname}
\theoremstyle{plain}
\newtheorem{cor}[thm]{\protect\corollaryname}
\theoremstyle{remark}
\newtheorem{claim}[thm]{\protect\claimname}
\theoremstyle{plain}
\newtheorem{question}[thm]{\protect\questionname}
\theoremstyle{plain}
\newtheorem{prop}[thm]{\protect\propositionname}

\usepackage{amsfonts}
\usepackage{nicefrac}
\usepackage{amscd}
\usepackage{a4wide}
\usepackage{url}

\makeatother

\providecommand{\acknowledgementname}{Acknowledgement}
\providecommand{\claimname}{Claim}
\providecommand{\corollaryname}{Corollary}
\providecommand{\definitionname}{Definition}
\providecommand{\factname}{Fact}
\providecommand{\mainthmname}{Main Theorem}
\providecommand{\propositionname}{Proposition}
\providecommand{\questionname}{Question}
\providecommand{\remarkname}{Remark}
\providecommand{\theoremname}{Theorem}

\begin{document}
\global\long\def\code#1{\ulcorner#1\urcorner}%
\global\long\def\p{\mathbf{p}}%
\global\long\def\q{\mathbf{q}}%
\global\long\def\C{\mathfrak{C}}%
\global\long\def\SS{\mathcal{P}}%
 
\global\long\def\pr{\operatorname{pr}}%
\global\long\def\image{\operatorname{im}}%
\global\long\def\otp{\operatorname{otp}}%
\global\long\def\dec{\operatorname{dec}}%
\global\long\def\suc{\operatorname{suc}}%
\global\long\def\pre{\operatorname{pre}}%
\global\long\def\qe{\operatorname{qf}}%
 
\global\long\def\ind{\operatorname{ind}}%
\global\long\def\Nind{\operatorname{Nind}}%
\global\long\def\lev{\operatorname{lev}}%
\global\long\def\Suc{\operatorname{Suc}}%
\global\long\def\HNind{\operatorname{HNind}}%
\global\long\def\minb{{\lim}}%
\global\long\def\concat{\frown}%
\global\long\def\cl{\operatorname{cl}}%
\global\long\def\tp{\operatorname{tp}}%
\global\long\def\id{\operatorname{id}}%
\global\long\def\cons{\left(\star\right)}%
\global\long\def\qf{\operatorname{qf}}%
\global\long\def\ai{\operatorname{ai}}%
\global\long\def\dtp{\operatorname{dtp}}%
\global\long\def\acl{\operatorname{acl}}%
\global\long\def\nb{\operatorname{nb}}%
\global\long\def\limb{{\lim}}%
\global\long\def\leftexp#1#2{{\vphantom{#2}}^{#1}{#2}}%
\global\long\def\intr{\operatorname{interval}}%
\global\long\def\atom{\emph{at}}%
\global\long\def\I{\mathfrak{I}}%
\global\long\def\uf{\operatorname{uf}}%
\global\long\def\ded{\operatorname{ded}}%
\global\long\def\Ded{\operatorname{Ded}}%
\global\long\def\Df{\operatorname{Df}}%
\global\long\def\Th{\operatorname{Th}}%
\global\long\def\eq{\operatorname{eq}}%
\global\long\def\Aut{\operatorname{Aut}}%
\global\long\def\ac{ac}%
\global\long\def\DfOne{\operatorname{df}_{\operatorname{iso}}}%
\global\long\def\modp#1{\pmod#1}%
\global\long\def\sequence#1#2{\left\langle #1\,\middle|\,#2\right\rangle }%
\global\long\def\set#1#2{\left\{  #1\,\middle|\,#2\right\}  }%
\global\long\def\Diag{\operatorname{Diag}}%
\global\long\def\Nn{\mathbb{N}}%
\global\long\def\mathrela#1{\mathrel{#1}}%
\global\long\def\twiddle{\mathord{\sim}}%
\global\long\def\mathordi#1{\mathord{#1}}%
\global\long\def\Qq{\mathbb{Q}}%
\global\long\def\dense{\operatorname{dense}}%
\global\long\def\Rr{\mathbb{R}}%
 
\global\long\def\cof{\operatorname{cof}}%
\global\long\def\tr{\operatorname{tr}}%
\global\long\def\treeexp#1#2{#1^{\left\langle #2\right\rangle _{\tr}}}%
\global\long\def\x{\times}%
\global\long\def\forces{\Vdash}%
\global\long\def\Vv{\mathbb{V}}%
\global\long\def\Uu{\mathbb{U}}%
\global\long\def\tauname{\dot{\tau}}%
\global\long\def\ScottPsi{\Psi}%
\global\long\def\cont{2^{\aleph_{0}}}%
\global\long\def\MA#1{{MA}_{#1}}%
\global\long\def\rank#1#2{R_{#1}\left(#2\right)}%
\global\long\def\cal#1{\mathcal{#1}}%

\def\Ind#1#2{#1\setbox0=\hbox{$#1x$}\kern\wd0\hbox to 0pt{\hss$#1\mid$\hss} \lower.9\ht0\hbox to 0pt{\hss$#1\smile$\hss}\kern\wd0} 
\def\Notind#1#2{#1\setbox0=\hbox{$#1x$}\kern\wd0\hbox to 0pt{\mathchardef \nn="3236\hss$#1\nn$\kern1.4\wd0\hss}\hbox to 0pt{\hss$#1\mid$\hss}\lower.9\ht0 \hbox to 0pt{\hss$#1\smile$\hss}\kern\wd0} 
\def\nind{\mathop{\mathpalette\Notind{}}} 

\global\long\def\ind{\mathop{\mathpalette\Ind{}}}%
\global\long\def\Age{\operatorname{Age}}%
\global\long\def\lex{\operatorname{lex}}%
\global\long\def\len{\operatorname{len}}%

\global\long\def\dom{\operatorname{Dom}}%
\global\long\def\res{\operatorname{res}}%
\global\long\def\alg{\operatorname{alg}}%
\global\long\def\dcl{\operatorname{dcl}}%
 
\global\long\def\nind{\mathop{\mathpalette\Notind{}}}%
\global\long\def\average#1#2#3{Av_{#3}\left(#1/#2\right)}%
\global\long\def\Ff{\mathfrak{F}}%
\global\long\def\mx#1{Mx_{#1}}%
\global\long\def\maps{\mathfrak{L}}%

\global\long\def\Esat{E_{\mbox{sat}}}%
\global\long\def\Ebnf{E_{\mbox{rep}}}%
\global\long\def\Ecom{E_{\mbox{com}}}%
\global\long\def\BtypesA{S_{\Bb}^{x}\left(A\right)}%
\global\long\def\DenseTrees{T_{dt}}%

\global\long\def\init{\trianglelefteq}%
\global\long\def\fini{\trianglerighteq}%
\global\long\def\Bb{\cal B}%
\global\long\def\Lim{\operatorname{Lim}}%
\global\long\def\Succ{\operatorname{Succ}}%

\global\long\def\SquareClass{\cal M}%
\global\long\def\leqstar{\leq_{*}}%
\global\long\def\average#1#2#3{Av_{#3}\left(#1/#2\right)}%
\global\long\def\cut#1{\mathfrak{#1}}%
\global\long\def\NTPT{\text{NTP}_{2}}%
\global\long\def\Zz{\mathbb{Z}}%
\global\long\def\TPT{\text{TP}_{2}}%
\global\long\def\supp{\operatorname{supp}}%

\global\long\def\OurSequence{\mathcal{I}}%
\global\long\def\SUR{SU}%
\global\long\def\ShiftGraph#1#2{Sh_{#2}\left(#1\right)}%
\global\long\def\ShiftGraphHalf#1{\cal G_{#1}^{\frac{1}{2}}}%
\global\long\def\SymShiftGraph#1#2{Sh_{#2}^{sym}\left(#1\right)}%

\title{On uniform definability of types over finite sets for NIP formulas}
\author{Shlomo Eshel and Itay Kaplan}
\thanks{The second author would like to thank the Israel Science Foundation
for partial support of this research (grants no. 1533/14 and 1254/18).}
\address{Itay Kaplan, Einstein Institute of Mathematics, Hebrew University
of Jerusalem, 91904, Jerusalem Israel.}
\email{kaplan@math.huji.ac.il}
\address{Shlomo Eshel, Einstein Institute of Mathematics, Hebrew University
of Jerusalem, 91904, Jerusalem Israel.}
\email{Shlomo.Eshel@mail.huji.ac.il}
\begin{abstract}
Combining two results from machine learning theory we prove that a
formula is NIP if and only if it satisfies uniform definability of
types over finite sets (UDTFS). This settles a conjecture of Laskowski.
\end{abstract}

\maketitle

\section{Introduction}

Let $L$ be any language and let $T$ be any $L$-theory. An $L$-formula
$\varphi(x,y)$ has \emph{uniform definability of types over finite
sets} (\emph{UDTFS}) in $T$ iff there is a formula $\psi(y,z)$ which
uniformly (in any model of $T$) defines $\varphi$-types over finite
sets of size $\geq2$ (see Definition \ref{def:(UDTFS)}).  If $\varphi$
has UDTFS, then for any finite $A\subseteq M^{y}\vDash T$, the number
of $\varphi$-types over $A$ is bounded by $\left|A\right|^{\left|z\right|}$,
which immediately implies that $\varphi$ has finite VC-dimension
in $T$, i.e., $\varphi$ is NIP  (if $\varphi$ shatters a finite
set $A$, then the number of $\varphi$-types over $A$ is exponential
in $\left|A\right|$). This raises the question, asked by Laskowski,
of whether these two notions (UDTFS and NIP) are equivalent. Note
that in that case, this also implies the Sauer-Shelah lemma in the
sense of counting types (see \cite[Chapter II, Theorem 4.10(4)]{Sh:c}).
See also the discussion in \cite{LivniSimon}.

This question was first addressed in \cite{MR2610477} where it was
proved assuming that $T$ is weakly o-minimal. Later, \cite{MR2963018}
extended this result to dp-minimal theories. Finally, in \cite[Theorem 15]{ArtemPierreII}
it was proved in the level of the theory $T$: a (complete) theory
is NIP iff every formula has UDTFS. They actually proved something
stronger: in NIP theories, every formula has uniform \emph{honest}
definitions. See Section \ref{subsec:Honest-definitions} below for
the definition.

The main theorem in this paper solves Laskowski's question (and thus
answers all the questions in the final paragraph of \cite{MR2963018}).
\begin{mainthm}
The following are equivalent for an $L$-theory $T$ and an $L$-formula
$\varphi(x,y)$.
\begin{enumerate}
\item $\varphi$ is NIP in $T$ (i.e., NIP in any completion of $T$).
\item $\varphi$ has UDTFS in $T$.
\end{enumerate}
\end{mainthm}

(see Theorem \ref{thm:Main} below.)

The proof has two ingredients, both from machine learning theory.

The first is \cite{MR3549531} which proves the existence of\emph{
sample compression schemes }for concept classes of finite VC-dimension
$d$ whose sizes are bounded in terms of $d$ (answering a question
of Littlestone and Warmuth). Roughly speaking, this result says that
there is some number $k$ depending only on $d$ such that for any
finite set of labeled examples (\emph{concepts}), it is possible to
recover our knowledge on any concept by considering a specific subset
of size $k$ and some bounded extra information. We do not use the
result but rather its proof, and most importantly the proof of Claim
3.1 from there, which we translate to our language.

The second ingredient is \cite{k-isolationExponential} where an upper
bound for the\emph{ recursive teaching dimension} (\emph{RTD}) is
given for concept classes of finite VC-dimension $d$ (the bound in
\cite{k-isolationExponential} is exponential in $d$ and was later
improved to a quadratic bound in \cite{k-isolationQuadratic}). Roughly
speaking this means that there is some number $t$ (depending only
on $d$) such that every concept can be identified by at most $t$
samples according to the recursive teaching model. See Fact \ref{fact:k-isolation}
for a precise statement which follows by reading the definitions.
This results translates in our language to the existence of $\varphi$-types
which are isolated by their restriction to a set of bounded size (see
Corollary \ref{cor:what we use from k-isolation}). This result (or
rather, its proof) will be used in a forthcoming work with Martin
Bays and Pierre Simon which deals with ``compressible'' types in
NIP theories.

Despite the fact that our proof is based on these two results, we
do not need to define any of the machine learning notions mentioned
above so that the proof can be read by anyone with a basic understanding
of model theory.
\begin{rem}
The first ingredient was known by experts for quite some time now
(it was brought to our attention by Pierre Simon in 2015). It was
known that it alone implies UDTFS assuming that the theory has definable
Skolem functions (see Section \ref{subsec:Having-definable-Skolem}).

We became aware of the second ingredient during the aforementioned
work with Martin Bays and Pierre Simon, thanks to Nati Linial who
answered our question about it.

The paper is organized as follows: Section \ref{sec:Proof} contains
all the preliminaries and the proof of the main theorem. Section \ref{sec:Open-questions}
contains some open questions.
\end{rem}

\begin{acknowledgement*}
We thank the anonymous referee for their remarks, and specifically
for suggesting to add Remark \ref{rem:proof of cct} which certainly
helps the presentation. 
\end{acknowledgement*}

\section{\label{sec:Proof}Proof of the main theorem}

Throughout fix a language $L$; all formulas, theories and structures
will be in $L$. For notation, we use $\varphi(x,y)$ to denote a
formula $\varphi$ with a partition of (perhaps a superset of) the
free variables of $\varphi$. When $x$ is a tuple of variables and
$A$ is a set, we write $A^{x}$ to denote the tuples of length $|x|$
from $A$; alternatively, one may think of $A^{x}$ as the set of
functions from the tuple $x$ to $A$. 
\begin{defn}
Let $M$ be a $L$-structure and let $\varphi(x,y)$ be a formula.
Suppose that $p(x)$ is a $\varphi$-type over a set $A\subseteq M^{y}$.
We say that a formula $\psi(y)$ (not necessarily over $\emptyset$)
\emph{defines $p$} if for all $b\in A$ we have $M\vDash\psi(b)\Leftrightarrow\varphi(x,b)\in p(x)$.
\end{defn}

\begin{defn}
\label{def:(UDTFS)}(UDTFS) Let $T$ be an $L$-theory. We say that
$\varphi(x,y)$ has \emph{uniform definability of types over finite
sets in $T$ }(\emph{UDTFS}) if there exists a formula $\psi(y,z)$
such that for every $M\vDash T$ and all finite sets $A\subseteq M^{y}$
with $\left|A\right|\geq2$ the following holds: for every $p(x)\in S_{\varphi}(A)$
there exist $c\in A^{z}$ such that $\psi(y,c)$ defines $p(x)$.
\end{defn}

\begin{defn}
(VC-dimension) Let $X$ be a set and $\mathcal{F}\subseteq\SS(X)$.
Given $A\subseteq X$, we say that it is shattered by $\mathcal{F}$
if for every $S\subseteq A$ there is $F\in\mathcal{F}$ such that
$F\cap A=S$. A family $\mathcal{F}$ is said be a \emph{VC-class}
if there is some $n<\omega$ such that no subset of $X$ of size $n$
is shattered by $\mathcal{F}$. In this case the \emph{VC-dimension
of $\mathcal{F}$}, denoted by $VC(\mathcal{F})$, is the smallest
integer $n$ such that no subset of $X$ of size $n+1$ is shattered
by $\mathcal{F}$.
\end{defn}

\begin{defn}
\label{def:NIP in T}($\varphi$ is NIP in $T$) Suppose that $T$
is an $L$-theory and that $\varphi\left(x,y\right)$ is a formula.
Say that \emph{$\varphi(x,y)$ is NIP in $T$ }if for every $M\vDash T$,
the family $\set{\varphi(M,a)}{a\in M^{y}}$ is a VC-class. This is
equivalent to saying that $\varphi$ is NIP in any completion of $T$.
 
\end{defn}

\begin{rem}
\label{rem:phi NIP in T iff bounded VC-dim}Note that $\varphi$ is
NIP in $T$ iff there is a bound $n<\omega$ such that for every $M\models T$,
the VC-dimension of the family $\set{\varphi(M,a)}{a\in M^{y}}$ is
bounded by $n$ (and this is first-order expressible). This follows
by compactness. Denote the minimal such $n$ by $VC_{T}(\varphi)$
or just $VC(\varphi)$ if $T$ is clear from the context. From now
on we will just write NIP for NIP in $T$ when $T$ is clear from
the context. 
\end{rem}

\begin{fact}
\cite[Remark 9]{Ad}\label{fact:Every-Boolean-combination of NIP is NIP}Suppose
that $T$ is an $L$-theory. The family of formulas $\varphi\left(x,y\right)$
which are NIP is closed under finite Boolean combinations. Moreover,
(it follows from the proof that) if $\varphi_{0},\varphi_{1}$ are
such and $\psi$ is a Boolean combination of $\varphi_{0},\varphi_{1}$,
then $VC_{T}(\psi)$ is bounded in terms of $VC_{T}(\varphi_{0}),VC_{T}(\varphi_{1})$.
\end{fact}

\begin{fact}
\cite[Proposition 2]{Ad}\label{fact:dual}If $\varphi(x,y)$ is NIP
then so is $\varphi^{opp}(y,x)$ where $\varphi^{opp}$ is the formula
$\varphi$ but with the partition of variables switched. As above,
$VC_{T}(\varphi^{opp})$ is bounded in terms of $VC_{T}(\varphi)$,
in fact $VC_{T}(\varphi^{opp})<2^{VC_{T}(\varphi)+1}$. 
\end{fact}

\begin{fact}
\label{fac:(VC-theorem)}\cite[Corollary 6.9]{simon2015guide}(VC-theorem;
the existence of $\epsilon$-approximations) For any $d<\omega$ and
$0<\epsilon$ there is some $N=N(d,\epsilon)<\omega$ such that for
any finite set $X$, for any $C\subseteq\SS(X)$ of VC-dimension $\leq d$
and every probability measure $\mu$ on $X$ there exists a multiset
$Y\subseteq X$ of size $\left|Y\right|\leq N$ such that for all
$s\in C$, $\left|\mu(s)-\frac{\left|Y\cap s\right|}{\left|Y\right|}\right|\leq\epsilon$.
\end{fact}

The next fact is the ``second ingredient'' mentioned in the introduction.
\begin{fact}
\label{fact:k-isolation}\cite[Theorem 3]{k-isolationExponential}\cite[Theorem 6]{k-isolationQuadratic}
For all $n<\omega$ there is some $t=t\left(n\right)$ such that if
$X$ is a finite set and $\cal F\subseteq\SS(X)$ is a family of VC-dimension
$\leq n$, then there is some $F\in\cal F$ and $X_{0}\subseteq X$
of size $|X_{0}|\leq t$ such that for all $F'\in\cal F$, $F=F'$
iff $X_{0}\cap F'=X_{0}\cap F$.
\end{fact}

\begin{rem}
\label{rem:proof of cct}Even though we do not need the proof of Fact
\ref{fact:k-isolation}, we will provide a sketch, based on the proof
of \cite[Lemma 4]{k-isolationExponential}.  The proof is by induction
on $n$. For $n=0$ this is trivial since then $\left|\cal F\right|=1$
(we can choose $t=0$). Assume the statement is true for $n-1$. Let
$k=2^{n}\left(n-1\right)+1$ and let $t\left(n\right)=t\left(n-1\right)+k$.
We may assume that $\left|X\right|\geq k$. Find $b,c$ such that
$c\subseteq b\subseteq X$, $\left|b\right|=k$ and $\cal F_{b,c}=\set{F\in\cal F}{F\cap b=c}$
is nonempty of minimal size. We claim that $VC\left(\cal F_{b,c}\right)<n$
(considering $\cal F_{b,c}$ as a subset of $\SS\left(X\right)$).
If not, then there is some $e\subseteq X$ such that $\left|e\right|=n$
which is shattered by $\cal F_{b,c}$. Note that $e\cap b=\emptyset$.
For every $x\in b$ there is some subset $e_{x}\subseteq e$ such
that if $F_{1},F_{2}\in\cal F$ and $F_{1}\cap e=e_{x}=F_{2}\cap e$
then $x\in F_{1}$ iff $x\in F_{2}$: indeed, otherwise for some $x\in b$,
$\left\{ x\right\} \cup e$ is shattered by $\cal F$, contradicting
$VC\left(\cal F\right)\leq n$. By pigeonhole (and by the choice of
$k$), we can find some $e'\subseteq e$ and a subset $b'\subseteq b$
such that $\left|b'\right|=n$ and if $F_{1},F_{2}\in\cal F$ are
such that $F_{1}\cap e=e'=F_{2}\cap e$ then $F_{1}\cap b'=F_{2}\cap b'$
(which must be $c\cap b'$). Let $b''=\left(b\cup e\right)\backslash b'$
and $c''=\left(c\cup e'\right)\backslash b'$. Then $\emptyset\neq\cal F_{b'',c''}\subsetneq\cal F_{b,c}$,
a contradiction. Thus $t\left(n\right)=t\left(n-1\right)+k$ works:
given $F,X_{0}$ as in the fact for $\cal F_{b,c}$, the pair $F,X_{0}\cup b$
will work. 
\end{rem}

\begin{cor}
\label{cor:what we use from k-isolation} Let $m,n<\omega$. Then
there is some $k=k(n,m)$ such that if $T$ is an $L$-theory, $\varphi\left(x,y\right)$
is NIP and $VC(\varphi)\leq n$ then for every $M\vDash T$ and finite
$A\subseteq M^{y}$ the following holds:

If $\chi(x)=\bigwedge_{i<m}\varphi(x,a_{i})^{\varepsilon_{i}}$ where
for every $i<m$, $a_{i}\in A$ and $\varepsilon_{i}<2$ (in general,
$\varphi^{0}=\varphi$ and $\varphi^{1}=\neg\varphi$) is consistent,
then there is a type $p_{0}\in S_{\varphi}(A)$ and $A_{0}\subseteq A$
of size $|A_{0}|\leq k$ such that $p_{0}|_{A_{0}}\vdash p_{0}\vdash\chi(x)$.
\end{cor}

\begin{proof}
Let $k=k(n,m)=t(2^{n+1}-1)+m$ where $t(n)$ is from Fact \ref{fact:k-isolation}.
Let $A'=A\backslash\set{a_{i}}{i<m}$, and consider the family $\cal F=\set{\varphi(b,A')}{b\vDash\chi}$.
Then the VC-dimension of $\cal F$ is bounded by the dual VC-dimension
of $\varphi$ which is $\leq2^{n+1}-1$. By Fact \ref{fact:k-isolation}
there is some $b\in M^{x}$ and some $X_{0}\subseteq A$ of size $\leq t(2^{n+1}-1)$
such that $\chi(b)$ holds and for all $b'\vDash\chi$, $\varphi(b,A')=\varphi(b',A')$
iff $\varphi(b,X_{0})=\varphi(b',X_{0})$. Let $p_{0}=\tp_{\varphi}\left(b/A\right)$
and let $A_{0}=X_{0}\cup\set{a_{i}}{i<m}$. Obviously $p_{0}\vdash\chi(x)$
and if $b'\vDash p_{0}|_{A_{0}}$ then $b'\vDash\chi$ and since $\varphi(b,X_{0})=\varphi(b',X_{0})$,
it follows that $\varphi(b,A')=\varphi(b',A')$, so that $\varphi(b,A)=\varphi(b',A)$,
i.e., $b'\vDash p_{0}$.
\end{proof}
The following theorem is the main result of this paper.
\begin{thm}
\label{thm:Main}Fix a theory $T$ and a formula $\varphi(x,y)$.
Then the following are equivalent:
\begin{enumerate}
\item $\varphi(x,y)$ is NIP.
\item There is an integer $K(\varphi)$ such that for every model $M\vDash T$
and for every finite nonempty $A\subseteq M^{y}$ and $p\in S_{\varphi}(A)$,
there is a formula $\psi(y,z)$ where $z=\sequence{z_{i}}{i<K(\varphi)}$
which is a finite disjunction of  formulas each of the form 
\[
\exists d_{0},...,d_{m-1}\bigwedge_{t<m}\bigwedge_{i\in s_{t}}\varphi(d_{t},z_{i})^{\varepsilon_{i,t}}\wedge\bigwedge_{t\in s}\varphi(d_{t},y)
\]
where $m\leq K(\varphi)$, $s\subseteq m$ and for every $t<m$, $s_{t}\subseteq K(\varphi)$
and $\varepsilon_{i,t}<2$ for all $i\in s_{t}$ such that $\psi(y,a)$
define $p$ for some $a\in A^{z}$.
\item There is an integer $K(\varphi)$ such that for every model $M\vDash T$
and for every finite nonempty $A\subseteq M^{y}$ and $p\in S_{\varphi}(A)$,
there is a formula $\psi(y,z)$ where $z=\sequence{z_{i}}{i<K(\varphi)}$
which is a finite disjunction of  formulas each of the form 
\[
\forall d_{0},...,d_{m-1}\bigwedge_{t<m}\bigwedge_{i\in s_{t}}\varphi(d_{t},z_{i})^{\varepsilon_{i,t}}\rightarrow\bigwedge_{t\in s}\varphi(d_{t},y)
\]
where $m\leq K(\varphi)$, $s\subseteq m$ and for every $t<m$, $s_{t}\subseteq K(\varphi)$
and $\varepsilon_{i,t}<2$ for all $i\in s_{t}$ such that $\psi(y,a)$
define $p$ for some $a\in A^{z}$.
\item $\varphi(x,y)$ have UDTFS in $T$.
\end{enumerate}
\end{thm}

We start with a proof of (1) implies (2), (3), so assume (1). By Remark
\ref{rem:phi NIP in T iff bounded VC-dim} and (1), the VC-dimension
of the family $\set{\varphi(M,b)}{b\in M^{y}}$ is bounded by some
constant integer $VC(\varphi)$.

Towards a proof of (2) and (3), for the next few claims we fix $M\vDash T$,
a finite $A\subseteq M^{y}$ and $p\in S_{\varphi}(A)$. Fix also
$c\vDash p$ (exists in $M$ since $A$ is finite). 

We will produce integers $N,J,k$ depending only on $VC(\varphi)$
and not on $M,A$ or $p$. From these we will be able to construct
the defining formula $\psi$.
\begin{defn}
For $n<\omega$, we say that $f:A^{n}\times2^{n}\longrightarrow M$
is an \emph{$(n,\varphi)$-Skolem function} if for every $(\overline{a},\overline{\varepsilon})\in A^{n}\times2^{n}$,
if there is $b\in M^{x}$ such that $M\vDash\mathord{\mathord{\bigwedge}_{i<n}}\varphi(b,a_{i})^{\varepsilon_{i}}$
then $M\vDash\mathord{\bigwedge}_{i<n}\varphi(f(\overline{a},\overline{\varepsilon}),a_{i})^{\varepsilon_{i}}$.
The function $f$ is a \emph{$\varphi$-Skolem function} if it is
an $(n,\varphi)$-Skolem function for some $n<\omega$.

For every $\overline{a}\in A^{n}$ and $b\in M^{x}$, let $\bar{\varepsilon}^{(\bar{a},b)}\in2^{n}$
be the unique tuple $\overline{\varepsilon}$ for which $M\vDash\mathord{\mathord{\bigwedge}_{i<n}}\varphi(b,a_{i})^{\varepsilon_{i}}$
(when $n=1$ we write $\varepsilon^{(a,b)}$).

\end{defn}

\begin{claim}
\label{claim:N}There is some integer $N=N(VC(\varphi))$ such that
for every probability measure $\mu$ on $A$, there is a tuple $\overline{a}\in A^{\leq N}$
such that for every $(|\overline{a}|,\varphi)$-Skolem function $f$
and for every $b\in M^{x}$ we have that $\mu(\set{a\in A}{\varphi(b,a)\leftrightarrow\varphi(f(\overline{a},\overline{\varepsilon}^{(\overline{a},b)}),a)})\geq\frac{2}{3}$.
\end{claim}

\begin{proof}
Consider the set $\mathcal{S}=\set{\varphi(c_{1},A)\mathrela{\triangle}\varphi(c_{2},A)}{c_{1},c_{2}\in M^{x}}\subseteq\SS(A)$.

By Fact \ref{fact:Every-Boolean-combination of NIP is NIP}, the formula
$\psi(xx',y)=\varphi(x,y)\mathrela{\triangle}\varphi(x',y)$  is
NIP  and has a finite VC-dimension which is moreover bounded in terms
of $VC(\varphi)$. By Fact \ref{fact:dual} it follows that the same
is true for the family $\cal S$. Let $N$ be the number provided
by the VC-theorem for this bound and $\epsilon=\frac{1}{3}$. It then
follows that for any $\mu$ as above there is a multiset $E\subseteq A$
with $|E|\leq N$ such that $|\mu(S)-\frac{|S\cap E|}{|E|}|\leq\frac{1}{3}$
for every $S\in\mathcal{S}$. Let $\overline{a}=\sequence{a_{i}}{i<n}\in A^{\leq N}$
be a tuple listing $E$. Now, fix $b\in M^{x}$ and a $\varphi$-Skolem
function $f$ as above. Let $S_{0}=\varphi(b,A)\mathrela{\triangle}\varphi(f(\overline{a},\overline{\varepsilon}^{(\overline{a},b)}),A)$.
Then for every $i<n$:
\[
a_{i}\in\varphi(b,A)\Leftrightarrow M\vDash\varphi(b,a_{i})\Leftrightarrow
\]
\[
M\vDash\varphi(f(\overline{a},\overline{\varepsilon}^{(\overline{a},b)}),a_{i})\Leftrightarrow a_{i}\in\varphi(f(\overline{a},\overline{\varepsilon}^{(\overline{a},b)}),A),
\]
i.e., $E\cap S_{0}=\emptyset$. Therefore $|\mu(S_{0})|\leq\frac{1}{3}$
and we are done.
\end{proof}
\begin{rem}
In fact we will only use Claim \ref{claim:N} with $b=c$, i.e., considering
$\cal S'=\set{\varphi(c_{1},A)\mathrela{\triangle}\varphi(c,A)}{c_{1}\in M^{x}}$.
It is easy to see that the VC-dimension of $\cal S'$ equals the VC-dimension
of $\set{\varphi(c_{1},A)}{c_{1}\in M^{x}}$ (both shatter the same
subsets of $A$) so we do not really require Fact \ref{fact:Every-Boolean-combination of NIP is NIP}
for the rest, but the current form of the claim is a bit more uniform.
 
\end{rem}

The following claim is a translation of \cite[Claim 3.1]{MR3549531}.
It can actually be deduced from there but for the completeness of
the paper we chose to include the proof.
\begin{claim}
\label{claim:3.1}There is an integer $J=J(VC(\varphi))$ such that
for every tuple of $\varphi$-Skolem functions $\sequence{f_{n}}{n\leq N}$
where $f_{n}$ is an $(n,\varphi)$-Skolem function, there are tuples
$\overline{a_{0}},...,\overline{a_{m-1}}\in A^{\leq N}$ for $m\leq J$
such that for every $a\in A$ 
\[
\left|\set{t<m}{M\vDash(\varphi(c,a)\leftrightarrow\varphi(f_{|\overline{a_{t}}|}(\overline{a_{t}},\overline{\varepsilon}^{(\overline{a_{t}},c)}),a))}\right|>\frac{m}{2}.
\]
\end{claim}

\begin{proof}
Let $J=J(\varphi)$ be the number we get by applying the VC-theorem
on $VC(\varphi)$ and $\epsilon=\frac{1}{8}$.

Let $\sequence{f_{n}}{n\leq N}$ be a sequence of $\varphi$-Skolem
functions as in the claim.

By choice of $N$ from Claim \ref{claim:N}, for every probability
measure $\mu$ on $A$ there is a tuple $\overline{a}\in A^{\leq N}$
such that $\mu(\set{a\in A}{\varphi(c,a)\leftrightarrow\varphi(f_{|\overline{a}|}(\overline{a},\overline{\varepsilon}^{(\overline{a},c)}),a)})\geq\frac{2}{3}$.

For $\overline{a}\in A^{\leq N}$, let $h_{\overline{a}}=f_{|\overline{a}|}(\overline{a},\overline{\varepsilon}^{(\overline{a},c)})$.
Let $\mathcal{H}=\set{h_{\overline{a}}}{\overline{a}\in A^{\leq N}}$.
Enumerate $A=\sequence{a_{i}}{i<|A|}$ and $\cal H=\sequence{h_{j}}{j<|\cal H|}$,
and define an $|A|\times|\mathcal{H}|$-matrix $B$ by setting
\[
B_{i,j}=\begin{cases}
1 & M\vDash(\varphi(c,a_{i})\leftrightarrow\varphi(h_{j},a_{i}))\\
0 & \text{else}
\end{cases}.
\]
For any probability measure $\mu$ on $A$, treat $\mu$ as a distribution
vector of length $|A|$. For $j<|\cal H|$, let $v^{j}\in2^{|\cal H|}$
be such that $v_{j'}^{j}=1$ iff $j=j'$ and let $\widetilde{A}_{j}=\set{a_{i}\in A}{B_{i,j}=1}$.
We have that (when both $v^{j}$ and $\mu$ are treated as column
vectors):

\[
\mu^{t}Bv^{j}=\mu^{t}\overset{\text{the }j\text{'th- column}}{\begin{pmatrix}|\\
B_{j}\\
|
\end{pmatrix}}=\sum_{i<|A|,B_{i,j}=1}\mu(a_{i})=\sum_{a\in\widetilde{A}_{j}}\mu(a)=
\]
\[
\mu(\set{a\in A}{M\vDash(\varphi(c,a)\leftrightarrow\varphi(h_{j},a))})\geq\frac{2}{3}.
\]

In particular we have $\underset{\mu\in\triangle^{|A|}}{\min}\underset{\nu\in\triangle^{|\mathcal{H}|}}{\max}\mu^{t}B\nu\geq\frac{2}{3}$
where (in general) $\triangle^{n}$ is the set of all probability
measures on the set $n=\{0,\ldots,n-1\}$.
\begin{fact}
\label{fact:(Von-Neumann's-minimax-theorem):} \cite{MR1512486} (von
Neumann's minimax theorem) for every matrix $M\in\mathbb{R}^{n\times m}$
\[
\underset{q\in\triangle^{n}}{\min}\underset{p\in\triangle^{m}}{\max}q^{t}Mp=\underset{p\in\triangle^{m}}{\max}\underset{q\in\triangle^{n}}{\min}q^{t}Mp.
\]
\end{fact}

By von Neumann's minimax theorem it follows that $\underset{\nu\in\triangle^{|\mathcal{H}|}}{\max}\underset{\mu\in\triangle^{|A|}}{\min}\mu^{t}B\nu\geq\frac{2}{3}$.
Therefore, there is some $\nu\in\triangle^{|\mathcal{H}|}$ such that
for every $\mu\in\triangle^{|A|}$, $\mu^{t}B\nu\geq\frac{2}{3}$.
For $i<|A|$, let $u^{i}\in2^{|A|}$ be such that $u_{i'}^{i}=1$
iff $i=i'$.  In particular, we have that
\[
\left(u^{i}\right)^{t}B\nu\geq\frac{2}{3}\Rightarrow\underset{\text{the \ensuremath{i}'th row}}{\begin{pmatrix}- & B_{i} & -\end{pmatrix}}\nu\geq\frac{2}{3}\Rightarrow\sum_{h_{\overline{a}}\in\widetilde{\mathcal{H}_{i}}}\nu(h_{\overline{a}})\geq\frac{2}{3}
\]
where $\widetilde{\mathcal{H}_{i}}=\set{h_{j}\in\mathcal{H}}{B_{i,j}=1}$.
Thus, for every $a\in A$ we have 
\[
\nu(\set{h_{\overline{a}}\in\mathcal{H}}{M\vDash(\varphi(c,a)\leftrightarrow\varphi(h_{\overline{a}},a))})\geq\frac{2}{3}.
\]
Consider the sets $S^{\varepsilon}=\set{\varphi(\mathcal{H},a)^{\varepsilon}}{a\in M^{y}}$
for $\varepsilon<2$. Recall that the VC-dimension of $S^{0}$ is
bounded by $VC(\varphi)$. By the choice of $J$ we can apply the
VC-theorem on $S^{0},\nu$ and $\epsilon=\frac{1}{8}$ and get that
there is a multiset $F=\set{h_{\overline{a_{i}}}}{i<m}\subseteq\mathcal{H}$
such that $m\leq J$ and $\left|\nu(\varphi(\mathcal{H},a))-\frac{|F\cap\varphi(\mathcal{H},a)|}{|F|}\right|\leq\frac{1}{8}$
for every $a\in M^{y}$.
\begin{rem}
\label{remark:X=00005Cbackslash Y}For every probability measure $p$
on a finite set $X$ and every finite multiset $F\subseteq X$ if
$\left|p(Y)-\frac{|F\cap Y|}{|F|}\right|\leq\epsilon$ for some $Y\subseteq X$
and $\epsilon>0$ then 
\[
\left|p(X\backslash Y)-\frac{|F\cap(X\backslash Y)|}{|F|}\right|\leq\epsilon.
\]
\end{rem}

By Remark \ref{remark:X=00005Cbackslash Y} it follows that for every
$a\in M$ 
\[
\left|\nu(\neg\varphi(\mathcal{H},a))-\frac{|F\cap\neg\varphi(\mathcal{H},a)|}{|F|}\right|\leq\frac{1}{8}.
\]
Now note that for every $a\in M^{y},b\in M^{x}$ we have that $D_{a,b}\in S^{0}$
or $D_{a,b}\in S^{1}$, where $D_{a,b}=\set{h_{\overline{a}}\in\mathcal{H}}{M\vDash(\varphi(b,a)\leftrightarrow\varphi(h_{\overline{a}},a))}$
(if $M\vDash\varphi(b,a)$ then $D_{a,b}=\varphi(\mathcal{H},a)$
and $D_{a,b}=\neg\varphi(\mathcal{H},a)$ otherwise). So for every
$D_{a,b}$ we have $\left|\nu(D_{a,b})-\frac{|F\cap D_{a,b}|}{|F|}\right|\leq\frac{1}{8}$
which implies that $\frac{|F\cap D_{a,b}|}{m}\geq\nu(D_{a,b})-\frac{1}{8}$.
As $\nu(D_{a,c})\geq\frac{2}{3}$ for all $a\in A$ (by the choice
of $\nu$), it then follows that for every $a\in A$ we have $\frac{|F\cap D_{a,c}|}{m}\geq\frac{2}{3}-\frac{1}{8}>\frac{1}{2}$.
Thus, for all $a\in A$, $|F\cap D_{a,c}|>\frac{m}{2}$ i.e., $|\set{t<m}{M\vDash(\varphi(c,a)\leftrightarrow\varphi(h_{\overline{a_{t}}},a))}|>\frac{m}{2}.$
\end{proof}
Now, let us finish the proof of (1) implies (2), (3) from Theorem
\ref{thm:Main}.
\begin{proof}[Proof of (1) implies (2),(3)]
By Corollary \ref{cor:what we use from k-isolation}, there is an
integer $k=k(VC(\varphi),N)$ such that for all $n\leq N$ and $(\overline{a},\overline{\varepsilon})\in A^{n}\times2^{n}$
if $\chi_{(\overline{a},\overline{\varepsilon})}(x)=\bigwedge_{i<|\overline{a}|}\varphi(x,a_{i})^{\varepsilon_{i}}$
is consistent then there are $p_{0}^{(\overline{a},\overline{\varepsilon})}\in S_{\varphi}(A)$
and $a_{0}^{(\overline{a},\overline{\varepsilon})},...,a_{k-1}^{(\overline{a},\overline{\varepsilon})}\in A$
such that, letting $A_{0}^{(\overline{a},\overline{\varepsilon})}=\set{a_{i}^{(\overline{a},\overline{\varepsilon})}}{i<k}$,
we have that $p_{0}^{(\overline{a},\overline{\varepsilon})}\vdash\chi_{(\overline{a},\overline{\varepsilon})}(x)$
and $p_{0}^{(\overline{a},\overline{\varepsilon})}|_{A_{0}}\vdash p_{0}^{(\overline{a},\overline{\varepsilon})}$.
Now, for every such $\overline{a},\overline{\varepsilon}$, define
$f_{n}^{*}(\overline{a},\overline{\varepsilon})=h$ for some $h\vDash p_{0}^{(\overline{a},\overline{\varepsilon})}.$
It follows that for all $n\leq N$, $f_{n}^{*}$ is an $(n,\varphi)$-Skolem
function.

By Claim \ref{claim:3.1}, for every sequence of $\varphi$-Skolem
functions $\sequence{f_{n}}{n\leq N}$ (such that for all $n\leq N$,
$f_{n}$ is an $(n,\varphi)$-Skolem function), there is some $m\leq J$
and tuples $\overline{a_{0}},...,\overline{a_{m-1}}\in A^{\leq N}$
such that for every $a\in A$ 
\[
\left|\set{t<m}{\varphi(x,a)\in p\Leftrightarrow M\vDash\varphi(f_{|\overline{a_{t}}|}(\overline{a_{t}},\overline{\varepsilon}^{(\overline{a_{t}},c)}),a)}\right|>\frac{m}{2}.
\]
In particular this true for $\sequence{f_{n}^{*}}{n\leq N}$, hence
we get that there are $h_{0},...,h_{m-1}\in M^{x}$ (namely $h_{t}=f_{|\overline{a_{t}}|}^{*}(\overline{a_{t}},\overline{\varepsilon}^{(\overline{a_{t}},c)})$)
such that:
\begin{itemize}
\item For every $t<m$, $\tp_{\varphi}(h_{t}/A)$ is $k$-isolated, i.e.,
there are $a_{0}^{t},...,a_{k-1}^{t}\in A$ such that for every $h'\in M^{x}$
if for every $i<k$, $M\vDash\varphi(h_{t},a_{i}^{t})\leftrightarrow\varphi(h',a_{i}^{t})$
then $\tp_{\varphi}(h_{t}/A)=\tp_{\varphi}(h'/A)$, and
\item We have that for every $a\in A$, 
\[
\left|\set{t<m}{\varphi(x,a)\in p\Leftrightarrow M\vDash\varphi(h_{t},a)}\right|>\frac{m}{2}.
\]
\end{itemize}
We claim that the following formula (which is over $A$)
\[
\psi(y)=\exists d_{0},...,d_{m-1}\bigwedge_{t<m}\bigwedge_{i<k}\varphi(d_{t},a_{i}^{t})^{\varepsilon^{(a_{i}^{t},h_{t})}}\wedge\left(\left|\set{t<m}{\varphi(d_{t},y)}\right|>\frac{m}{2}\right)
\]
defines the type $p$.

Indeed: first assume that $\varphi(x,a)\in p$. Then by the second
bullet we have that 
\[
M\vDash\bigwedge_{t<m}\bigwedge_{i<k}\varphi(h_{t},a_{i}^{t})^{\varepsilon^{(a_{i}^{t},h_{t})}}\wedge\left(\left|\set{t<m}{\varphi(h_{t},a)}\right|>\frac{m}{2}\right).
\]

And hence $\psi\left(a\right)$ holds.

Now suppose $M\vDash\psi(a)$ for $a\in A$. Then there are $h'_{0},...,h'_{m-1}$
witnessing that $\psi$ holds. But since $h'_{t}$ agrees with $h_{t}$
on $a_{0}^{t},...,a_{k}^{t}$ for every $t<m$, it follows by the
first bullet that $\tp_{\varphi}(h_{t}/A)=\tp_{\varphi}(h_{t}'/A)$
and in particular $\tp_{\varphi}(h_{t}/a)=\tp_{\varphi}(h_{t}'/a)$
which implies that $|\set{t<m}{\varphi(h_{t},a)}|>\frac{m}{2}$. Towards
a contradiction, if $\varphi(x,a)\notin p$ then by the second bullet
it follows that $M\vDash\neg\varphi(h_{t},a)$ for more than half
the $t$'s but there are more than half the $t$'s for which $M\vDash\varphi(h_{t},a)$
which is impossible.

To see that (2)  holds just note that 
\[
\psi(y)=\exists d_{0},...,d_{m-1}\bigwedge_{t<m}\bigwedge_{i<k}\varphi(d_{t},a_{i}^{t})^{\varepsilon^{(a_{i}^{t},h_{t})}}\wedge\left(\left|\set{t<m}{\varphi(d_{t},y)}\right|>\frac{m}{2}\right)
\]
is equivalent to

\[
\bigvee_{s\subseteq m,\left|s\right|>m/2}\exists d_{0},...,d_{m-1}\bigwedge_{t<m}\bigwedge_{i<k}\varphi(d_{t},a_{i}^{t})^{\varepsilon^{(a_{i}^{t},h_{t})}}\wedge\bigwedge_{t\in s}\varphi(d_{t},y).
\]
This proves (2), where $K(\varphi)$ can be easily recovered from
the proof using $N(VC(\varphi))$, $J(VC(\varphi))$ and $k(VC(\varphi),N(VC(\varphi))$
--- a rough estimate is $K(\varphi)\leq J(VC(\varphi))\cdot k(VC(\varphi),N(VC(\varphi))$
(assuming both are positive which we may).

In order to show that (3) holds note that for every $a\in A$ and
every $s\subseteq m$ we have:

\[
M\vDash\exists d_{0},...,d_{m-1}\bigwedge_{t<m}\bigwedge_{i<k}\varphi(d_{t},a_{i}^{t})^{\varepsilon^{(a_{i}^{t},h_{t})}}\wedge\bigwedge_{t\in s}\varphi(d_{t},a)
\]
 if and only if
\[
M\vDash\forall d_{0},...,d_{m-1}\bigwedge_{t<m}\bigwedge_{i<k}\varphi(d_{t},a_{i}^{t})^{\varepsilon^{(a_{i}^{t},h_{t})}}\to\bigwedge_{t\in s}\varphi(d_{t},a).
\]

Indeed, this follows easily from the first bullet as above (and the
fact that for all $t<m$ the formula $\bigwedge_{i<k}\varphi(x,a_{i}^{t})^{\varepsilon^{(a_{i}^{t},h_{t})}}$
is consistent).

From the above we get that for every $a\in A$ we have that $M\vDash\psi(a)\Leftrightarrow\psi'(a)$
where 
\[
\psi'(y)=\bigvee_{s\subseteq m,\left|s\right|>m/2}\forall d_{0},...,d_{m-1}\bigwedge_{t<m}\bigwedge_{i<k}\varphi(d_{t},a_{i}^{t})^{\varepsilon^{(a_{i}^{t},h_{t})}}\to\bigwedge_{t\in s}\varphi(d_{t},y)
\]

which gives (3).
\end{proof}
Finally, let us finish the proof.
\begin{proof}[Proof of Theorem \ref{thm:Main}]
 Assuming either (2) or (3), it is clear that there are at most finitely
many formulas of the forms described there (note that the number of
disjuncts in the formula is bounded). Thus by coding finitely many
formulas into one as in \cite[Lemma 2.5]{MR2963018} we get (4).

(4) implies (1) easily follows from type-counting argument as mentioned
in the introduction. From (4) it follows that for any finite set $A\subseteq M^{y}$
for any $M\vDash T$, $|S_{\varphi}(A)|\leq|A|^{|z|}$ where $\psi(y,z)$
uniformly defines $\varphi$-types. On the other hand, if the VC-dimension
of $\varphi$ is not bounded in $T$, then we can find finite sets
$A$ as above such that $|S_{\varphi}(A)|$ is exponential in $|A|$,
contradiction.
\end{proof}
\begin{rem}
Note that from the proof of Theorem \ref{thm:Main} one can extract
an explicit bound for the length of the variable $z$ in the defining
formula $\psi(y,z)$ that suits $\varphi(x,y)$ (as in Definition
\ref{def:(UDTFS)}) which depends only on $VC(\varphi)$. An inspection
of the proof can give a better description of the formula given in
the formulation of Theorem \ref{thm:Main}. Any improvement on the
bounds we used will automatically induce a shorter formula.
\end{rem}

\section{\label{sec:Open-questions}Open questions}

\subsection{NIP defining formula}
\begin{question}
Can we improve Theorem \ref{thm:Main} as to ensure that the defining
formula $\psi(y,z)$ is itself NIP?
\end{question}

Let us justify this question in two instances, one is when $\varphi$
is stable in $T$, and the other is when $T$ has definable Skolem
functions.

\subsubsection{Stable formulas}

Suppose that $T$ is a theory. As in Definition \ref{def:NIP in T},
a formula $\varphi(x,y)$ is\emph{ stable in $T$} if it is stable
in any completion of $T$ (see e.g., \cite[Theorem 8.2.3]{TentZiegler};
we omit ``in $T$'' for the rest of the discussion). If $\varphi(x,y)$
is stable  then a much stronger result than UDTFS holds: there is
some formula $\psi(y,z)$ such that for any $M\vDash T$ and any $A\subseteq M^{y}$,
every $\varphi$-type $p\in S_{\varphi}(A)$ is definable by a formula
of the form $\psi(y,a)$ for some $a\in A^{z}$. This can be deduced
using the 2-rank as in \cite[Chapter II, Theorem 2.12 (3)]{Sh:c}.
However, the formula that this proof gives involves quantifiers so
it is not obviously stable. Using the apparatus of non-forking extensions
one can overcome this as we now explain (this is probably well-known,
but it is not stated explicitly like this as far as we know).

First,  for every $M\vDash T$, any $\varphi$-type over $M$ is
definable over $M$ by some formula $\psi(y,m)$ where $\psi(y,z)$
is a Boolean combination of (positive) instances of $\varphi^{opp}$,
see e.g., \cite[Lemma 2.2(i)]{PillayGeometric}. In particular, $\psi(y,z)$
is itself stable  (see e.g., \cite[Lemma 2.1]{PillayGeometric}).

If $A\subseteq M\vDash T$ is any set and $p\in S_{\varphi}(A)$ (in
the notation of the previous section we should have written $p\in S_{\varphi}(A^{y})$,
but we ignore this for now), then $p$ has a non-forking extension
$q\in S_{\varphi}(M)$ which is definable over $\acl^{\eq}(A)$ via
some formula $\theta(y,a)$ (see e.g., \cite[Lemma 2.7]{PillayGeometric}).
By the first paragraph $q$ is also definable over $M$ via a Boolean
combination of instances of $\varphi^{opp}$, so that $\theta(y,a)$
is itself equivalent in $M^{\eq}$ to a Boolean combination of instances
of $\varphi^{opp}$.
\begin{prop}
\label{prop:equivalent to stable then stable}Suppose that $M\vDash T$,
$\psi(x,y)$ and $\theta(x,z)$ are formulas and $a\in M^{y},a'\in M^{z}$.
Suppose that $\psi(x,a)$ is equivalent to $\theta\left(x,a'\right)$
in $M$ and that $\theta(x,z)$ is stable. Then there is a formula
$\psi'(x,y)$ such that $\psi'(x,a)$ is equivalent to $\psi(x,a)$
in $M$ and $\psi'(x,y)$ is stable.
\end{prop}

\begin{proof}
Let  $\tau(yz)=\forall x(\psi(x,y)\leftrightarrow\theta(x,z))$ so
that $M\vDash\tau(aa')$. Let $\psi'(x,y)=\psi(x,y)\land\exists z\tau(yz)$.
Then $\psi'(x,a)$ is equivalent to $\psi(x,a)$ in $M$ and $\psi'$
is stable: suppose towards a contradiction that $\sequence{b_{i}a_{i}}{i<\omega}$
is an indiscernible sequence which witnesses that $\psi'$ has the
order property ($\psi'(b_{i},a_{j})$ iff $i\leq j$) in some model
$N\vDash T$. Then, since e.g., $\psi'(b_{0},a_{0})$ holds, by indiscernibility
it follows that for every $j<\omega$, there exists some $a_{j}'$
such that $\tau(a_{j}a_{j}')$ holds. Hence we get that $i\leq j$
iff $\theta(b_{i},a_{j}')$, contradicting the stability of $\theta$.
\end{proof}
Continuing our discussion from above, Proposition \ref{prop:equivalent to stable then stable}
implies that $q$ is definable over $\acl^{\eq}(A)$ via some formula
$\theta'(y,a)$ such that $\theta'(y,z)$ is stable.

 Let $\chi(z,c)\in\tp(a/A)$ be an algebraic formula of minimal (finite)
size so that if $a'\vDash\chi$ then $a'\equiv_{A}a$. Let $\beta(y,c)=\exists z\chi(z,c)\land\theta'(y,z)$.
Then $\beta(y,c)$ is equivalent (in $M$) to $\bigvee_{a'\vDash\chi(z,c)}\theta'(y,a')$
and $\bigvee_{i<m}\theta'(y,z_{i})$ is stable as a Boolean combination
of stable  formulas (where $m$ is the size of $\chi(M,c)$). Hence
 by Proposition \ref{prop:equivalent to stable then stable} there
is some formula $\beta'(y,w)$ such that $\beta'(y,c)$ defines $p$
and $\beta'(y,w)$ is stable.

Now, if we had started with $A\subseteq M^{y}$, then we could have
let $A'$ be the set of all elements appearing in some tuple from
$A$, and do the same process. 

From all this we got that:
\begin{cor}
\label{cor:stable formula non-uniform}If $\varphi(x,y)$ stable,
$M\vDash T$, $A\subseteq M^{y}$ and $p\in S_{\varphi}(A)$ then
$p$ is $A$-definable via a stable formula $\psi(y,z)$.
\end{cor}

And a uniform version:
\begin{cor}
Suppose that $T$ is any $L$-theory and that $\varphi(x,y)$ is stable.
Then there is a stable formula $\psi(y,z)$ such that for all $M\vDash T$,
all $A\subseteq M^{y}$ with $|A|\geq2$ and all $p\in S_{\varphi}(A)$
there is some $a\in A^{z}$ such that $\psi(y,a)$ defines $p$.
\end{cor}

\begin{proof}
By compactness, as in the proof of \cite[Chapter II, Theorem 2.12(3)]{Sh:c}.

First we show that ({*}) there are finitely many stable formulas
which work for every $p\in S_{\varphi}(A)$. Indeed, assume not and
add a new predicate $P(y)$ to the language and consider the partial
type $\Gamma(x)$ in the language $L\cup\{P\}$ consisting of formulas
$\theta_{\psi}(x)$ where 
\[
\theta_{\psi}(x)=\neg\exists z\in P(\forall y\in P(\varphi(x,y)\leftrightarrow\psi(y,z)))
\]
 for every stable formula $\psi(y,z)$ (where $z$ is a tuple of
copies of $y$).

By our assumption towards a contradiction, for all stable formulas
$\psi_{0}(y,z_{0}),\ldots,\psi_{k-1}(y,z_{k-1})$ there is some $M\vDash T$,
$A\subseteq M^{y}$ and a type $p\in S_{\varphi}(A)$ which is not
definable by any of the formulas $\psi_{0},\ldots,\psi_{k-1}$. This
means that $\Gamma(x)$ is consistent. By compactness there is a model
$M\models T$, $A\subseteq M$ and $a\in M^{x}$ such that $\tp_{\varphi}(a/A)$
is not definable by any stable formula $\psi(y,z)$. But this contradicts
Corollary \ref{cor:stable formula non-uniform}. This shows ({*}).

Now using the standard coding trick as in \cite[Lemma 2.5]{MR2963018}
we can code finitely many formulas $\psi_{0},\ldots,\psi_{k-1}$ into
one formula $\psi$. We leave it to the reader to make sure that if
all the formulas $\psi_{0},\ldots,\psi_{k-1}$ are stable, then so
is $\psi$.
\end{proof}

\subsubsection{\label{subsec:Having-definable-Skolem}Having definable Skolem functions}
\begin{prop}
If $T$ is any theory with definable Skolem functions and $\varphi(x,y)$
is an NIP formula, then there is a formula $\psi(y,z)$ which uniformly
defines $\varphi$-types over finite sets and such that $\psi(y,z)$
is itself NIP.
\end{prop}

\begin{proof}
By inspecting the proof of Theorem \ref{thm:Main} (1) implies (2),
one sees that in the case when $T$ has definable Skolem functions,
then we would not have to use Corollary \ref{cor:what we use from k-isolation}
at all. Instead, we could define the $\varphi$-type $p\in S_{\varphi}(A)$
by 
\[
\varphi(x,a)\in p\Leftrightarrow\left|\set{t<m}{\varphi(f_{|\overline{a_{t}}|,\overline{\varepsilon}^{(\overline{a_{t}},c)}}(\overline{a_{t}}),a)}\right|>\frac{m}{2},
\]
where $c\models p$ and the $\bar{a}_{t}$'s are provided by Claim
\ref{claim:3.1} for the definable Skolem whose existence we assume,
denoted by $f_{n,\bar{\varepsilon}}$ (i.e., $f_{n,\overline{\varepsilon}}(\overline{a})$
returns some element $b$ satisfying $\bigwedge_{i<n}\varphi(x,a_{i})^{\varepsilon_{i}}$
if such exists). In other words, the formula $\psi(y,b)$ defining
$p$ is a Boolean combination of instances of formulas of the form
$\varphi(f(b),y)$ for $b$ a tuple from $A$ and $f$ some definable
function.

In general, when $\varphi(x,y)$ is an NIP formula and $f$ is some
definable function, then $\psi(z,y)=\varphi(f(z),y)$ is also NIP
 (if $\set{\psi(M,b)}{b\in M^{y}}$ shatters some set $A\subseteq M^{z}$
then $\set{\varphi(M,b)}{b\in M^{y}}$ shatters the image of $A$
under $f$ which has the same size as $A$). This means (by Fact
\ref{fact:Every-Boolean-combination of NIP is NIP}) that $\psi(y,z)$
is NIP.

As in the stable case, coding finitely many NIP formulas into one
gives an NIP formula, so we are done.
\end{proof}

\subsection{\label{subsec:Honest-definitions}Honest definitions}
\begin{defn}
\cite[Definition 3.16 and Remark 3.14]{simon2015guide} Work in some
theory $T$. Suppose that $\varphi(x,y)$ is a formula, $A\subseteq M^{y}$
some set and $p\in S_{\varphi}(A)$. Say that a formula $\psi(y,z)$
(over $\emptyset$) is an \emph{honest definition} of $p$ if for
every finite $A_{0}\subseteq A$ there is some $b\in A^{z}$ such
that for all $a\in A$, if $\psi(a,b)$ holds then $\varphi(x,a)\in p$
and for all $a\in A_{0}$ the other direction holds: if $\varphi(x,a)\in p$
then $\psi(a,b)$ holds.
\end{defn}

It is proved in \cite[Theorem 6.16]{simon2015guide}, \cite[Theorem 11]{ArtemPierreII}
that if $T$ is NIP then for every $\varphi(x,y)$ there is a formula
$\psi(y,z)$ that serves as an honest definition for any type in $S_{\varphi}(A)$
provided that $|A|\geq2$.
\begin{question}
Is this true only assuming that $\varphi(x,y)$ is NIP?
\end{question}

Note that for the proof in \cite[Theorem 6.16]{simon2015guide}, \cite[Theorem 11]{ArtemPierreII},
only a very mild assumption of NIP is actually needed. See \cite[Remark 16]{ArtemPierreII}.

\bibliographystyle{alpha}
\bibliography{common}

\begin{thebibliography}{HWLW17}

\bibitem[Adl08]{Ad}
Hans Adler.
\newblock An introduction to theories without the independence property.
\newblock {\em Archive of Mathematical Logic}, 2008.
\newblock accepted.

\bibitem[CCT16]{k-isolationExponential}
Xi~Chen, Yu~Cheng, and Bo~Tang.
\newblock On the recursive teaching dimension of {VC} classes.
\newblock In Daniel~D. Lee, Masashi Sugiyama, Ulrike von Luxburg, Isabelle
  Guyon, and Roman Garnett, editors, {\em Advances in Neural Information
  Processing Systems 29: Annual Conference on Neural Information Processing
  Systems 2016, December 5-10, 2016, Barcelona, Spain}, pages 2164--2171, 2016.

\bibitem[CS15]{ArtemPierreII}
Artem Chernikov and Pierre Simon.
\newblock Externally definable sets and dependent pairs {II}.
\newblock {\em Trans. Amer. Math. Soc.}, 367(7):5217--5235, 2015.

\bibitem[Gui12]{MR2963018}
Vincent Guingona.
\newblock On uniform definability of types over finite sets.
\newblock {\em J. Symbolic Logic}, 77(2):499--514, 2012.

\bibitem[HWLW17]{k-isolationQuadratic}
Lunjia Hu, Ruihan Wu, Tianhong Li, and Liwei Wang.
\newblock Quadratic upper bound for recursive teaching dimension of finite {VC}
  classes.
\newblock In Satyen Kale and Ohad Shamir, editors, {\em Proceedings of the 30th
  Conference on Learning Theory, {COLT} 2017, Amsterdam, The Netherlands, 7-10
  July 2017}, volume~65 of {\em Proceedings of Machine Learning Research},
  pages 1147--1156. {PMLR}, 2017.

\bibitem[JL10]{MR2610477}
H.~R. Johnson and M.~C. Laskowski.
\newblock Compression schemes, stable definable families, and o-minimal
  structures.
\newblock {\em Discrete Comput. Geom.}, 43(4):914--926, 2010.

\bibitem[LS13]{LivniSimon}
Roi Livni and Pierre Simon.
\newblock Honest compressions and their application to compression schemes.
\newblock In Shai Shalev-Shwartz and Ingo Steinwart, editors, {\em COLT 2013 -
  The 26th Annual Conference on Learning Theory, June 12--14, 2013, Princeton
  University, NJ, USA}, volume~30 of {\em JMLR Proceedings}, pages 77--92.
  JMLR.org, 2013.

\bibitem[MY16]{MR3549531}
Shay Moran and Amir Yehudayoff.
\newblock Sample compression schemes for {VC} classes.
\newblock {\em J. ACM}, 63(3):Art. 21, 10, 2016.

\bibitem[Pil96]{PillayGeometric}
Anand Pillay.
\newblock {\em Geometric stability theory}, volume~32 of {\em Oxford Logic
  Guides}.
\newblock The Clarendon Press, Oxford University Press, New York, 1996.
\newblock Oxford Science Publications.

\bibitem[She90]{Sh:c}
Saharon Shelah.
\newblock {\em Classification theory and the number of nonisomorphic models},
  volume~92 of {\em Studies in Logic and the Foundations of Mathematics}.
\newblock North-Holland Publishing Co., Amsterdam, second edition, 1990.

\bibitem[Sim15]{simon2015guide}
P.~Simon.
\newblock {\em A Guide to NIP Theories}.
\newblock Lecture Notes in Logic. Cambridge University Press, 2015.

\bibitem[TZ12]{TentZiegler}
Katrin Tent and Martin Ziegler.
\newblock {\em A course in model theory}, volume~40 of {\em Lecture Notes in
  Logic}.
\newblock Association for Symbolic Logic, La Jolla, CA; Cambridge University
  Press, Cambridge, 2012.

\bibitem[vN28]{MR1512486}
J.~v.~Neumann.
\newblock Zur {T}heorie der {G}esellschaftsspiele.
\newblock {\em Math. Ann.}, 100(1):295--320, 1928.

\end{thebibliography}

\end{document}